\documentclass[10pt, oneside]{amsart}
\usepackage{amssymb,latexsym}
\usepackage{pstricks}
\usepackage{calc}

\def\emph#1{{\it #1}}
\newtheorem{thm}{Theorem}[section]
\newtheorem{cor}[thm]{Corollary}
\newtheorem{lem}[thm]{Lemma}

\theoremstyle{definition}
\newtheorem{defin}[thm]{Definition}
\newtheorem{rem}[thm]{Remark}
\newtheorem{exa}[thm]{Example}
\newtheorem{notation}[thm]{Notation}

\newtheorem{proposition}[thm]{Proposition}

\numberwithin{equation}{section}
\usepackage[margin=1.5in]{geometry}

\usepackage{euscript}

\newcommand{\Tmc}{T_{\!\rm mc}}
\newcommand{\Aut}{{\textrm{Aut}}}
\newcommand{\E}{\exists}
\newcommand{\A}{\forall}
\newcommand{\Exp}{{\textrm{Exp}}}

\newcommand{\imp}{\rightarrow}
\renewcommand{\iff}{\leftrightarrow}
\newcommand{\dom}{\mathop\textrm{dom}}
\newcommand{\range}{\mathop\textrm{rng}}
\newcommand{\IFF}{\Leftrightarrow}

\newcommand{\IMP}{\Rightarrow}

\newcommand{\acl}{\mathord{\textrm{acl}}}
\newcommand{\dg}{\mathord\textrm{diag}}
\newcommand{\tp}{\mathop\textrm{tp}}

\renewcommand{\phi}{\varphi}

\renewenvironment{itemize}
  {\begin{list}{$\triangleright$}{%
   \setlength{\parskip}{0mm}
   \setlength{\topsep}{.5\baselineskip}
   \setlength{\rightmargin}{0mm}
   \setlength{\listparindent}{0mm}
   \setlength{\itemindent}{0mm}
   \setlength{\labelwidth}{3ex}
   \setlength{\itemsep}{.5\baselineskip}
   \setlength{\parsep}{0mm}
   \setlength{\partopsep}{0mm}
   \setlength{\labelsep}{1ex}
   \setlength{\leftmargin}{\labelwidth+\labelsep}
   }}{%
   \end{list}\vspace*{-.5\baselineskip}}

\begin{document}
\title[Generic expansions]{Generic expansions of countable models}

\author[S. Barbina]{Silvia Barbina}
\thanks{The first author gratefully acknowledges support by the Commission of the European Union under contract MEIF-CT-2005-023302 `Reconstruction and generic automorphisms'.}
\address{Centro Internazionale per la Ricerca Matematica\\
Fondazione Bruno Kessler\\
Via Sommarive 14, Povo, 38123 Trento, Italy
}
\email{silvia.barbina@gmail.com}

\author[D. Zambella]{Domenico Zambella}
\address{Dipartimento di Matematica\\Universit\`a di Torino\\via Carlo Alberto 10, 10123 Torino, Italy}
\email{domenico.zambella@unito.it}

\date{}
\begin{abstract}
We compare two different notions of generic expansions of countable saturated structures. One kind of genericity is related to existential closure, another is defined via topological properties and Baire category theory. The second type of genericity was first formulated by Truss for automorphisms. We work with a later generalization,  due to Ivanov~\cite{ivanov}, to finite tuples of predicates and functions.

Let $N$ be a countable saturated model of some complete theory $T$, and let  $(N,\sigma)$ denote an expansion of $N$ to the signature $L_0$ which is a model of some universal theory $T_0$. We prove that when all e.c.\@ models of $T_0$ have the same existential theory, $(N,\sigma)$ is Truss generic if and only if $(N,\sigma)$ is an \textit{e-atomic} model. When $T$ is $\omega$--categorical and $T_0$ has a model companion $\Tmc$, the e-atomic models are simply the atomic models of $\Tmc$.
\end{abstract}

\maketitle

\section{Introduction}

In model theory there are two main notions of a generic automorphism of a structure. 
In some cases, the automorphisms that one obtains through these notions are similar enough that it is natural to ask whether, and how, they are related.

Let $T$ be a theory with quantifier elimination in a language $L$. Let $L_0=L\cup\{f\}$, where $f$ is a unary function symbol. Let $T_0$ be $T$ together with the sentences which say that $f$ is an automorphism. For a model $M$ of $T$ and $f \in \Aut(M)$, we say that $f$ is generic if $(M, f)$ is an existentially closed model of $T_0$ \cite{kikyo}.

This notion of genericity first appeared in~\cite{lascar}, where Lascar constructs some models of $T_0$ that have certain properties of universality and homogeneity. Later this became relevant to work on expansions of structures via an automorphism, mainly in the case of algebraically closed fields ~\cite{chhr, chpi}. In a series of papers (notably~\cite{chpi}, see also e.g.~\cite{kikyo},~\cite{kumac}, ~\cite{balshel}) conditions are given for $T_0$ to have a model companion $\Tmc$, describing the best case scenario where the e.c. models of $T_0$ are an elementary class.

A second notion of genericity was introduced by Truss in~\cite{truss}. An automorphism of a countable structure $M$ is Truss generic if its conjugacy class is comeagre in the canonical topology on the automorphism group $\Aut(M)$. More generally, a tuple $(f_1, \ldots, f_n) \in \Aut(M)^n$ is generic in this sense if $\{(f_1^g, \ldots, f_n^g) : g \in \Aut(M)\}$ is comeagre in the product space $\Aut(M)^n$. The intuition underlying this definition is that a generic automorphism should exhibit any finite behaviour that is consistent in the structure, modulo conjugacy. This is reminiscent of an existential closure condition, and suggests that a comparison with genericity \`a la Lascar is meaningful. Several related notions of generic automorphism are described --- and the relationship among some of them is investigated --- in~\cite{truss2}.

Truss generic automorphisms populate rather different habitats. Generic tuples are a useful tool in the two main techniques for reconstructing $\omega$--categorical structures from their automorphism group, namely, the small index property~\cite{lascar1} and Rubin's weak $\A\E$-interpretations ~\cite{rubin} (see e.g.~\cite{hhls} and~\cite{barmac} for specific applications of Truss generics). The existence of a comeagre conjugacy class is interesting in its own right: for an $\omega$-categorical structure $M$, it implies that $\Aut(M)$ cannot be written non trivially as a free product with amalgamation \cite{macthomas}. Ivanov~\cite{ivanov} isolates conditions under which a countable $\omega$-categorical structure has a Truss generic automorphism or tuple. In ~\cite{keros}, Kechris and Rosendal isolate conditions of this kind in the more general case of countable homogeneous structures and prove a wealth of topological consequences in Polish groups.

Ivanov generalises Truss genericity so that it applies to predicates, and indeed to arbitrary finite signatures~\cite{ivanov}. His work concerns generic expansions of $\omega$-categorical structures. One application is to the semantics of generalized quantifiers in the context of second-order logic. Lascar genericity, too, applies  to predicates: in~\cite{chpi} the authors show that for a complete $L$-theory $T$, $L_0=L\cup\{r\}$, where $r$ is a unary relation and $T_0=T$, $T_0$ has a model companion if and only if $T$ eliminates the $\E^\infty$ quantifier. Therefore it makes sense to extend the comparison to expansions of a structure by a finite tuple of predicates and functions, rather than simply by an automorphism.

In~\cite{ivanov} the structures considered are models of $\omega$-categorical theories. In ~\cite{keros} they are locally finite ultrahomogeneous structures. In order to provide a suitable framework for a comparison with generics \`a la Lascar, we require the base theory $T$ to be small and to have quantifier elimination. The latter
assumption is not essential but it streamlines a few definitions and it is standard in~\cite{chpi, kikyo, kikyoshelah}. We consider an expansion $T_0$ of $T$ in a language where finitely many predicate and function symbols are added. When $L_0 = L \cup \{ f \}$, where $f$ is a unary function symbol, and $T_0$ says that $f$ is an automorphism, the setting is as in~\cite{chpi, kikyo, kikyoshelah}. For our main results we require the e.c. models of $T_0$ to have the same existential theory (this is true in particular when $T_0$ has a model companion which is a complete theory).  While this assumption is more restrictive than in~\cite{keros} and, modulo $\omega$-categoricity, \cite{ivanov}, it allows us to replace Fra\"{\i}ss\'e limits with existentially closed models.

We work with a given countable saturated model $N\models T$ and we consider the set $\Exp(N,T_0)$ of expansions of $N$ that model $T_0$. We endow $\Exp(N,T_0)$ with the topology  in~\cite{ivanov}, a natural generalisation of the canonical topology on $\Aut(N)$, which makes  $\Exp(N,T_0)$ a Baire space.

In Section~\ref{baire} we define a subspace of $\Exp(N,T_0)$ which will later turn out to contain the Truss generic expansions. We define a set of `slightly saturated' expansions of $N$ which we call {\it smooth\/}. A smooth expansion of $N$ realizes all the types of the form \\
\centerline{($*$)\qquad $p_{\restriction L}(x) \cup \{\phi(x)\}$,}\\ where $p_{\restriction L}(x)$ is a type in the base language $L$ and $\phi(x)$ is a quantifier-free formula in the expanded language $L_0$. We prove that smooth expansions are a comeagre subset of $\Exp(N,T_0)$. The set of e.c. expansions is also comeagre, so that the smooth e.c. expansions form a Baire space in their own right.  

In Section~\ref{truss} we define {\it e-atomic} expansions. An e-atomic expansion is existentially closed, smooth, and only realizes $p(x)$ if $p_{\restriction\A}(x) \cup p_{\restriction\E}(x)$ is isolated by types of the form $\E y\, q(x,y)$, where $q(x,y)$ is as in ($*$). We show that the e-atomic expansions are exactly the expansions that are generic in the sense of~\cite{truss}. When $T$ is $\omega$-categorical and $\Tmc$ exists, this amounts to showing that the Truss generic expansions are the atomic models of $\Tmc$.

Our original purpose was to describe the role of Truss generic automorphisms among existentially closed models of $T_0$ when $T_0$ is as in ~\cite{chpi}. While both~\cite{ivanov} and~\cite{keros} work within the framework of amalgamation classes, our motivation led to a different approach and, occasionally, to some duplication of results in~\cite{ivanov} and~\cite{keros} under different assumptions. However, we have kept our version as it is functional to our comparison between notions of genericity. 

As remarked by the anonymous referee, some of our results appear with different terminology in~\cite{hodges1}, where the approach is that of Robinson forcing, so that `enforceable' corresponds to `comeagre' in our context. For a smoother comparison with~\cite{hodges1} one should take our $L$ to be empty and let $T$ be the theory of a pure infinite set. The Henkin constants play the role of the model $N$ in our context. Then the notion of $\exists$-atomic model translates to our \textit{e}-atomic. With this dictionary in mind, the reader may compare Lemma~\ref{lemmaec} with Corollary 3.4.3 of~\cite{hodges1} and Theorem~\ref{atomicrichcomeagre} with Theorem~4.2.6 (cf. also Theorem~5.1.6) of~\cite{hodges1}.

The first author is grateful to Alexander Berenstein for helpful initial remarks, and to Enrique Casanovas and Dugald Macpherson for useful conversation.

We thank the referee for several pivotal remarks and for pointing out some inaccuracies in earlier versions of the paper.

\section{Baire categories of first-order expansions}%
\label{baire}

Let $T$ be a complete theory with quantifier elimination in the countable language $L$.  Let $L_0$ be the language $L$ enriched with finitely many new relation and function symbols. We shall denote a structure of signature $L_0$ by a pair $(N,\sigma)$, where $N$ is a structure of signature $L$ and $\sigma$ is the interpretation of the symbols in $L_0 \smallsetminus L$. 

Let $T_0$ be any theory of signature $L_0$ containing $T$. We define
$$
\Exp(N,T_0)\ \ :=\ \ \Big\{\sigma \ :\ (N,\sigma)\models T_0\Big\}.
$$
We write $\Exp(N)$ for $\Exp(N,T)$.

There is a canonical topology on $\Exp(N)$, cf.~\cite{ivanov}, which makes it a Baire space. The purpose of this section is to define a subspace $Y$ of $\Exp(N)$, that of \textit{smooth, e-atomic} expansions, which is itself a Baire space and which in Section \ref{truss} proves significant for the relationship between Truss and Lascar generic expansions.

For a sentence $\phi$ with parameters in $N$ we define \emph{$[\phi]_N$} $:=$ $\{\sigma: (N,\sigma)\models\phi\}$. The topology on $\Exp(N)$ is generated by the open sets of the form $[\phi]_N$ where $\phi$ is quantifier-free. When $N$ is countable, this topology is completely metrizable: fix an enumeration $\{a_i:i\in\omega\}$ of $N$, define $d(\sigma,\tau)=2^{-n}$, where $n$ is the largest natural number such that for every tuple $a$ in $\{a_0,\dots,a_{n-1}\}$ and any symbols $r,f$ in $L_0\smallsetminus L$,
$$
a\in r^\sigma \Leftrightarrow a\in r^\tau\ \ \ \textrm{and}\ \ \  f^\sigma(a)= f^\tau(a),
$$
where $r^\sigma$ is the interpretation of $r$ in $(N, \sigma)$. When such an $n$ does not exist, $d(\sigma,\tau)=0$.

The reader may easily verify that this metric is complete. We check that it induces the topology defined above. Fix $n$ and $\tau$. Let $\phi$ be the conjunction of the formulas of the form $f a=b$ and $r a$ which hold in $(N,\tau)$ for some $b\in N$ and some tuple $a$ from $\{a_0,\dots,a_n\}$. Then
$$
[\phi]_N\ =\ \{\sigma : d(\sigma,\tau)< 2^{-n}\}.
$$ 
Conversely, let $\phi$ be a quantifier-free sentence with parameters in $N$, and take an arbitrary $\tau\in[\phi]_N$. Let $A$ be the set of parameters occurring in $\phi$. Let $n$ be large enough that
$$
\{ t^\tau(a) : a\subseteq A \textrm{ and } t \textrm{ is a subterm of a term appearing in } \phi  \} \ \ \subseteq\ \ \{a_0,\dots,a_{n-1}\}.
$$
Clearly $(N,\sigma)\models\phi$ for any $\sigma$ at distance $<2^{-n}$ from $\tau$ so
$$
\{\sigma : d(\sigma,\tau)< 2^{-n}\} \subseteq\ [\phi]_N\
$$ 
as required.

If $g:M\imp N$ is an isomorphism and $\sigma\in\Exp(M)$ we write $\sigma^g$ for the unique expansion of $N$ that makes $g:(M,\sigma)\imp (N,\sigma^g)$ an isomorphism. Explicitly, for every predicate $r$, every function $f$ in $L_0\smallsetminus L$, and every tuple $a\in N$,
$$
\llap{$(N,\sigma^g)\models r\, a$}\ \ \ \ \IFF\ \ \ \ \rlap{$(M,\sigma)\models r\, g^{-1}a$}
$$
$$
\llap{$(N,\sigma^g)\models f\, a= b$}\ \ \ \ \IFF\ \ \ \ \rlap{$(M,\sigma)\models f\,g^{-1}a =g^{-1}b$}
$$
We write $T_{0,\A}$ for the set of consequences of $T_0$ that are universal modulo $T$ (i.e. equivalent to a universal sentence in every model of $T$). Then
$$\Exp(N,T_0) \subseteq \Exp(N,T_{0,\A}) \subseteq \Exp(N).$$

\begin{notation}\label{notation1}
For the rest of this section we assume $T$ to be small and fix some $N$, a countable saturated model of $T$. We shall often avoid the distinction between the expansion $\sigma\in\Exp(N)$ and the model $(N,\sigma)$.
\end{notation}

\begin{lem}\label{closed=universal} 
Let $T_0$ be an arbitrary expansion of $T$ to the signature $L_0$. Then $\Exp(N,T_{0,\A})$ is the closure of $\Exp(N,T_0)$ in the above topology.
\end{lem}

\begin{proof} 
Let $\tau\in\Exp(N,T_{0,\A})$. We claim that $\tau$ is adherent to $\Exp(N,T_0)$.  Let $[\phi]_N$ be an arbitrary basic open set containing $\tau$. As $(N,\tau)$ models the universal consequences of $T_0$, there exists some $(N',\tau')\models T_0$ such that $(N,\tau)\subseteq(N',\tau')$. Let $A\subseteq N$ be the set of parameters occurring in $\phi$.  We may assume that $N'$ is countable and saturated (in $L$), therefore by q.e. in $L$ it is isomorphic to $N$ over $A$, so $[\phi]_N$ contains some element of $\Exp(N,T_0)$.

Conversely, suppose that $\tau\notin\Exp(N,T_{0,\A})$. Then for some parameter- and quantifier-free formula $\phi(x)$ we have $T_0\vdash\A x\,\phi(x)$ and $(N,\tau)\models \neg\phi(a)$. Then the open set $[\neg\phi(a)]_N$ separates $\tau$ from $\Exp(N,T_0)$.
\end{proof} 

\begin{notation}\label{notation2}
For the rest of this section we fix a theory $T_0$ that is universal modulo $T$, so that, by Lemma \ref{closed=universal}, $\Exp(N,T_0)$ is a closed subset of $\Exp(N)$, hence it is complete (as a metrizable space). If not otherwise specified, expansions $\sigma$, $\tau,$ etc.\@ range over $\Exp(N,T_0)$.\end{notation}

We say that $\sigma$ is \emph{existentially closed\/}, or e.c., if every quantifier-free $L_0$-formula with parameters in $N$ that has a solution in some $(U,\upsilon)$ such that $(N,\sigma)\subseteq(U,\upsilon)\models T_0$, has a solution in $(N,\sigma)$.

\begin{lem}\label{lemmaec}
The set of existentially closed expansions is comeagre in $\Exp(N,T_0)$.
\end{lem}

\begin{proof} 
Let $\psi(x)$ be a quantifier-free formula with parameters in $N$. We show that the following set is open dense:\medskip

\noindent($\star$)\hfil$\Big\{\sigma\ :\  (N,\sigma)\models\E x\psi(x) \Big\}
\cup
\Big\{\sigma\ :\ (U,\upsilon)\nvDash\E x\,\psi(x) \textrm{ for every } (N,\sigma)\subseteq(U,\upsilon)\models T_0\Big\}$.\medskip

The set of existentially closed expansions is the intersection of these sets as $\psi(x)$ ranges over the quantifier-free formulas of $L_0$. So the lemma follows.

It is clear that the first set in ($\star$) above is a union of basic open sets. For openness of the second set, suppose that $\sigma$ is such that there is no extension $(U, \upsilon) \vDash T_0 \cup \{ \E x\,\psi(x)\}$. Then $\mathrm{Diag}(N, \sigma) \cup T_0 \cup \{ \E x\,\psi(x)\}$ is inconsistent, hence by compactness there is $\chi \in  \mathrm{Diag}(N, \sigma)$ such that $T_0\models\chi \rightarrow \neg \E x\,\psi(x)$. Then $[\chi]_N$ is a neighbourhood of $\sigma$ contained in the second set in ($\star$).


For density, fix a basic open $[\phi]_N$ and consider the theory $T_0\cup\{\phi\wedge\E x\,\psi(x)\}$. If this theory is inconsistent then $[\phi]_N$ is contained in the second set in ($\star$). Otherwise it has a model $(U,\upsilon)$. As $U$ can be chosen to be countable and $L$-saturated, by q.e. in $L$ there is an $L$-isomorphism $g : U \mapsto N$ which fixes the parameters of $\phi\wedge\E x\,\psi(x)$. Then $\psi(x)$ has a solution in $(U^g,\upsilon^g)$, hence the first set in ($\star$) intersects $[\phi]_N$ in $\upsilon^g$.
\end{proof}

\begin{exa}\label{excompletesmall} Let $T$ be any complete small theory with quantifier elimination in the language $L$. Let $L_0\smallsetminus L$ contain only a unary relation symbol $r$ and let $T_0=T$. In~\cite{chpi} the authors prove that if $T$ eliminates the $\E^{\infty}$ quantifier, then $T_0$ has a model companion $\Tmc$. By Lemma~\ref{lemmaec}, $\Exp(N,\Tmc)$ is comeagre.
\end{exa}

\begin{exa}\label{PAPA} Let $T$ and $L$ be as in Example \ref{excompletesmall}. Let $L_0\smallsetminus L$ contain two unary function symbols $f$ and $f^{-1}$ and let $T_0$ be $T$ together with a sentence which says that $f$ is an automorphism with inverse $f^{-1}$. We need a symbol for the inverse of $f$ because we want $T_0$ to be universal. It is considerably more difficult than in Example \ref{excompletesmall} to find a condition which guarantees the existence of a model companion of $T_0$ ~\cite{balshel}. An important example where the model companion of $T_0$ exists is the case where $T$ is the theory of algebraically closed fields~\cite{chhr}. Then $\Tmc$ is also known as ACFA. Let $N$ be a countable algebraically closed field of infinite transcendence degree. By Lemma~\ref{lemmaec}, $\Exp(N,\Tmc)$ is comeagre.\end{exa}

\begin{defin}\label{defsmooth} We say that $\sigma$ is a \emph{smooth expansion\/} if $(N,\sigma)$ realizes every finitely consistent type of the form $p_{\restriction L}(x)\wedge\psi(x)$ where $\psi(x)$ is quantifier-free and $p_{\restriction L}(x)$ is a type in $L$ with finitely many parameters.\end{defin}

When $T$ is $\omega$-categorical, any expansion is smooth. For an example of an expansion that is {\it not\/} smooth, let $T$ be the theory of the algebraically closed fields of some fixed characteristic and let $N$ be an algebraically closed field of infinite transcendence degree. Expand $N$ by a relation $r(x)$ which holds exactly for the elements of $\acl(\varnothing)$. Then $(N, r)$ is not smooth.

\begin{lem}\label{gamesmooth}
The set of smooth expansions is comeagre in $\Exp(N,T_0)$.
\end{lem}
\begin{proof}
The set of smooth expansions is the intersection of sets of the form
$A\cup B$ where \medskip

\parbox{10ex}{\hfill A\ \ =\ \ }$\Big\{\sigma\ :\  (N,\sigma)\models\E
x\,[p_{\restriction L}(x)\wedge\psi(x)] \Big\}$,

\parbox{10ex}{\hfill B\ \ =\ \ }$\Big\{\sigma\ :\ p_{\restriction
L}(x)\wedge\psi(x) \textrm{ is not finitely consistent in
}(N,\sigma)\Big\}$,\medskip

\noindent and $p_{\restriction L}(x)\wedge\psi(x)$ range over the
types as in Definition~\ref{defsmooth}. As $T$ is small, there are
countably many of these sets. Let \medskip

\parbox{10ex}{\hfill C\ \ =\ \ }$\Big\{\sigma\ :\ \mathrm{Diag}(N,\sigma)\cup T_0\cup\big\{\E x\,[\xi(x)\wedge\psi(x)] : \xi(x)\in p_{\restriction L}(x)\big\}\textrm{ is inconsistent}\Big\}$, \medskip

\noindent and observe that $C \subseteq B$, so the lemma follows if we prove that $A\cup C$ is open
dense. 

For openness we argue as in Lemma~\ref{lemmaec}. For density, take
a basic open $[\phi]_N$ and consider the theory
$$
S\ =\ T_0\cup\{\phi\}\cup \Big\{\E x\,[\xi(x)\wedge\psi(x)]\ :\
\xi(x)\in p_{\restriction L}(x)\Big\}.
$$
If $S$ is inconsistent then $[\phi]_N$ is contained in $C$. Otherwise,
by compactness, $S$ has a model $(U,\upsilon)$ where $p_{\restriction
L}(x)\wedge\psi(x)$ has a solution $b$. As $U$ can be chosen to be
countable and $L$-saturated, by q.e.\@ there is an $L$-isomorphism
$g:U\to N$ that fixes the parameters of $p_{\restriction
L}(b)\wedge\phi\wedge \psi(b)$. Then $b$ is a solution of
$p_{\restriction L}(x)\wedge\psi(x)$ in $(N,\upsilon^g)$ as well, therefore $\upsilon^g\in
A\cap [\phi]_N$.
\end{proof}

We shall write $Y$ for the set of existentially closed smooth expansions of $N$. From Lemmas~\ref{lemmaec} and \ref{gamesmooth} we know that $Y$ is a comeagre subset of $\Exp(N,T_0)$.  We may regard $Y$ as a Polish space in its own right with the topology inherited from $\Exp(N,T_0)$. When $T$ is $\omega$-categorical, $Y$ is simply the set of e.c. models of $T_0$.

\section{Truss generic expansions}%
\label{truss}
The notation is as in~\ref{notation1} and~\ref{notation2}. When developing the results in this section we  originally had in mind the case when $T_0$ has a model companion $\Tmc$ which is a complete theory. These assumptions are motivated by the conditions described in~\cite{chpi} and they make the comparison between Truss generic and Lascar generic automorphisms rather neat. However, our results hold in the more general case where all existentially closed models of $T_0$ have the same existential theory, so this will be the underlying assumption.
If $\phi(x,y)$ is a quantifier-free formula in $L_0$ and $p(x,y)$ is a parameter-free type in $L$, then in every smooth model the infinitary formula $\E y\,[p(x,y)\wedge\phi(x,y)]$ is equivalent to a type. Infinitary formulas of this form are called \emph{existential quasifinite}.

Let $b$ be a finite tuple in $N$. For any $\alpha\in Y$ we define the \emph{1-diagram\/} of $\alpha$ at $b$
$$
\dg_{\restriction1}(\alpha, b)\ \ :=\ \ \big\{\phi(b)\ :\ \phi(x) \textrm{ is universal or existential and } (N,\alpha)\models \phi(b)\big\},
$$
and write \emph{$D_b$} for the set of 1-diagrams at $b$. On $D_b$ we define a topology whose basic open sets are of the form
$$
[\,\pi(b)\,]_D\ \ =\ \ \big\{\dg_{\restriction1}(\alpha,b)\ :\ (N,\alpha)\models \pi(b)\big\},
$$
where $\pi(x)$ is any existential quasifinite formula. When $\dg_{\restriction1}(\alpha,b)$ is an isolated point of $D_b$, we say that it is \emph{e-isolated} in $D_b$.

It is sometimes convenient to use the syntactic counterpart of $D_b$ which we now define. If $p(x)$ is a complete $L_0$-type, we write $p_{\restriction\A}(x)$, respectively $p_{\restriction\E}(x)$, for the set of universal, respectively existential, formulas in $p(x)$. We write $p_{\restriction1}(x)$ for $p_{\restriction\A}(x)\,\cup\, p_{\restriction\E}(x)$. We say that a type is realized in $Y$ if it is realized in some $(N,\sigma)$ with $\sigma\in Y$. Let $S_x^Y$ be the set of types of the form $p_{\restriction1}(x)$, where $p(x)$ is some complete parameter-free type realized in $Y$. On $S_x^Y$ define the topology where the basic open sets are of the form 
$$
[\,\pi(x)\,]_S\ \ =\ \ \Big\{ q_{\restriction1}(x) \ :\ \pi(x)\ \subseteq\ q(x)\Big\},
$$
where $\pi(x)$ is some existential quasifinite formula, and $q(x)$ ranges over the parameter-free types realized in $Y$. When $[\pi(x)]_S$ isolates $p_{\restriction 1}(x)$ in $S_x^Y$, we say that $p(x)$ is \emph{e-isolated} by $\pi(x)$.

\begin{lem}\label{homeomorphismDS} 
Let $b$ be a tuple in $N$ and let $p_{\restriction L}(x)$ be the parameter-free type of $b$ in the language $L$. There is a homeomorphism $h:D_b\ \imp\ [p_{\restriction L}(x)]_S$. For every existential quasifinte formula $\pi(x)$ containing $p_{\restriction L}(x)$, the image under $h$ of the set $[\pi(b)]_D$ is the set $[\pi(x)]_S$.
\end{lem}

\begin{proof} Let $h$ be the map that takes $\dg_{\restriction1}(\alpha,b)$ to the type \\
\centerline{$\{\phi(x)\,:\,\phi(b)\in\dg_{\restriction1}(\alpha,b)\}$. }

Note that, by q.e. in $L$,\@ this type contains $p_{\restriction L}(x)$. It is clear that $h$ maps $D_b$ injectively to $S_x^Y$. For surjectivity, let $q(x)$ be a complete parameter-free type realized in $Y$, say $(N,\sigma)\models q(a)$ for some $\sigma\in Y$, and suppose that $q_{\restriction1}(x)$ belongs to $[\pi(x)]_S$. As $p_{\restriction L}(x)\subseteq q(x)$, there is an isomorphism $g:N\imp N$ such that $g(a)=b$. Then $q_{\restriction1}(x)$ is the image of $\dg_{\restriction1}(\sigma^g,b)$ under $h$. This proves surjectivity.\end{proof}

From this fact it is clear that $\dg_{\restriction1}(\alpha, b)$ is e-isolated in $D_b$ if and only if $p(x)$, the parameter-free type of $b$ in $(N,\alpha)$, is e-isolated. The following lemma is also clear.

\begin{lem}\label{generalnonsence1} 
Let $p(x)$ be a complete parameter-free type realized in $Y$ and let $\pi(x)$ be an existential quasifinite formula such that  $p_{\restriction L}(x)\subseteq\pi(x)\subseteq p(x)$. Then the following are equivalent:
\begin{itemize}
\item[1.] $p(x)$ is e-isolated by $\pi(x)$;
\item[2.] $\pi(x)\ \models\ p_{\restriction1}(x)$ holds in every $\sigma\in Y$.
\end{itemize}
\end{lem}

\begin{defin}\label{new atomic rich} Let $\alpha\in Y$. We say that $(N,\alpha)$ is an e-atomic model, or that $\alpha$ is \emph{e-atomic}, if for all finite tuples $b$ in $N$  the 1-diagram $\dg_{\restriction1}(\alpha, b)$ is e-isolated. 
\end{defin}

The notion of e-atomic is close to Ivanov's notion of $(A,\E)$-atomic in~\cite{ivanov}, Section 2. However, the context is different and a circumstantial comparison is not straightforward.
When all e.c. models of $T_0$ have the same existential theory, any existential quasifinite formula is realized in all $\alpha \in Y$. Therefore in this case an e-atomic expansion $(N,\alpha)$ realizes $p_{\restriction1}(x)$ if and only if $p(x)$ is e-isolated.

\begin{rem} As remarked in Section~\ref{baire}, when $T$ is $\omega$-categorical, every expansion is smooth. In this case, if the model companion $T_{\rm mc}$ of $T_0$ exists, the e-atomic expansions are exactly the atomic models of  $T_{\rm mc}$.
\end{rem}

\begin{thm}\label{atomicrichconjugated} Suppose that $N \models T$ is countable and saturated and that all e.c. models of $T_0$ have the same existential theory. Then any two e-atomic expansions of $N$ are conjugate.
\end{thm}

\begin{proof}
Let $\alpha$ and $\beta$ be e-atomic. We prove the following claim: any finite $1$-elementary partial map $f:(N,\alpha)\imp(N,\beta)$ can be extended to an isomorphism, where a map is 1-elementary if it preserves existential and universal formulas. Since we assume all e.c. models to have the same existential theory, the empty map between existentially closed models is 1-elementary, so the theorem follows from the claim.

To prove the claim it suffices to show that for any finite tuple $b$ we can extend $f$ to some $1$-elementary map defined on $b$. The claim then follows by back and forth. Let $a$ be an enumeration of $\dom f$.  Then $\dg_{\restriction1}(\alpha, ab)$ is e-isolated in $D_b$, say by some existential quasifinite formula $\pi(v,x)$. Let $p(v,x) = \tp(a,b)$.  By fattening $\pi$ if necessary, we may assume that it contains $p_{\restriction L}(v, x)$. Since $\beta$ is smooth and $f$ is $1$--elementary, the type $\pi(fa,x)$ is realized in $\beta$, say by $c$. By Lemma~\ref{generalnonsence1}, $\pi(v,x)\ \models p_{\restriction1}(v,x)$ holds both in $\alpha$ and $\beta$, so $f \cup \{\langle b,c\rangle\}$ gives the required extension.
\end{proof}

\begin{thm}\label{atomicrichcomeagre} Suppose that $N \models T$ is countable and saturated and that all e.c. models of $T_0$ have the same existential theory.
If an e-atomic expansion of $N$ exists, then the set of e-atomic expansions is comeagre in $\Exp(N,T_0)$.
\end{thm}

\begin{proof} We prove that the set of e-atomic expansions is a dense $G_\delta$ subset of $Y$, hence comeagre in $\Exp(N,T_0)$.

To prove density, let $\psi(x)$ be a parameter- and quantifier-free formula. Let $a\in N$ be such that $\psi(a)$ is consistent with $T_0$. We show that $(N,\alpha)\models\psi(a)$ for some e-atomic $\alpha$. Write $p_{\restriction L}(x)$ for the parameter-free type of $a$ in the signature $L$.  Let $\beta$ be any e-atomic expansion and let $c$ be a realization of $p_{\restriction L}(x)\wedge\psi(x)$ in $(N,\beta)$. Let $g$ be an automorphism of $N$ such that $g(c)=a$. Then $\alpha:=\beta^g$ is the required expansion. Hence the set of e-atomic expansions is dense.

We now prove that the set of e-atomic expansions is a $G_\delta$ subset of $Y$. Let $b$ be a finite tuple and denote by $X_b$ the set of expansions $\alpha \in Y$ such that  $\dg_{\restriction1}(\alpha,b)$ is e-isolated.  It suffices to prove that $X_b$ is an open subset of $Y$.

Let $\alpha \in X_b$ and let $[\pi_\alpha(b)]_D$ be the basic open subset of $D_b$ that isolates $\dg_{\restriction1}(\alpha,b)$. We may assume $\pi_\alpha(b)$ has the form $\E y\,[\,p_{\alpha \restriction L}(b,y)\wedge \phi_\alpha(b,y)\,]$. So let $a_\alpha$ be a witness of the existential quantifier. We have that $Y\cap [\phi_\alpha(b, a_\alpha)]_N\subseteq X_b$. It follows that
$$
Y\ \cap\ \bigcup_{\alpha\in X_b}\  [\phi_\alpha(b, a_\alpha)]_N\ \  =\ \  X_b.
$$
Hence $X_b$ is an open subset of $Y$.\end{proof}

In~\cite{truss}, a notion of generic automorphisms is introduced and a number of examples are given of countable $\omega$-categorical structures that have generic automorphisms. The following definition, which appears in~\cite{ivanov}, generalizes the notion of generic automorphisms to arbitrary expansions.

\begin{defin} We say that an expansion $\tau$ is \emph{Truss generic} if $\{\tau^g\;:\; g\in\Aut(N)\}$ is a comeagre subset of $\Exp(N,T_0)$.
\end{defin}

\begin{rem}\label{atmostonecomeagre} There is at most one comeagre subset of $\Exp(N,T_0)$ of the form $\{\tau^g\;:\; g\in\Aut(N)\}$. This is because any two sets of this form are either equal or disjoint, and two comeagre sets in a Baire space have nonempty intersection.
\end{rem}

\begin{thm}\label{trussatomicrich} Suppose that $N \models T$ is countable and saturated and that all e.c. models of $T_0$ have the same existential theory. Let $\alpha$ be any expansion in $\Exp(N,T_0)$. Then the following are equivalent:
\begin{itemize}
\item[1.] $\alpha$ is e-atomic;
\item[2.] $\alpha$ is Truss generic.
\end{itemize}
\end{thm}

\begin{proof} Let $\alpha$ be e-atomic. By Theorem~\ref{atomicrichcomeagre}, the set $X$ of e-atomic expansions is comeagre. By Theorem~\ref{atomicrichconjugated}, and because $X$ is closed under conjugacy by elements of $\Aut(N)$, $X$ is of the form $\{\tau^g\,:\,  g\in\Aut(N)\}$ for any e-atomic $\tau$. By Remark~\ref{atmostonecomeagre}, $X$ is exactly the set of Truss generic expansions.

Conversely, let $\alpha$ be Truss generic. As smoothness and existential closure are guaranteed by Lemma~\ref{gamesmooth}, we only need to prove that $\alpha$ omits $p_{\restriction1}(x)$ for any complete parameter-free type $p(x)$ that is not e-isolated. It suffices to prove that the set of expansions in $Y$ that omit $p_{\restriction1}(x)$ is dense $G_\delta$ in $Y$, hence comeagre in $\Exp(N,T_0)$. Then some Truss generic expansion omits it and, as Truss generic expansions are conjugated, the same holds for $\alpha$.

Denote by $X_b$ the set of expansions in $Y$ that model $\neg p_{\restriction1}(b)$. The set of expansions in $Y$ that omit $p_{\restriction1}(x)$ is the intersection of $X_b$ as the tuple $b$ ranges over $N$. So it suffices to show that $X_b$ is open dense in $Y$.

First we prove density. Let $\psi(a,b)$ be a quantifier-free formula where $a$ and $b$ are disjoint tuples. We need to show that there is an expansion in $Y$ that models $\psi(a,b)\wedge\neg p_{\restriction1}(b)$. Let $q_{\restriction L}(z,x)$ be the parameter-free type of $a,b$ in the language $L$. Since $p(x)$ is not e-isolated, there is $\theta(x) \in p_{\restriction1}(x)$ such that $\psi(z,x) \wedge q_{\restriction L}(z,x) \wedge \neg \theta(x)$ is realised by some $a^\prime, b^\prime$ in some $\sigma \in Y$. There is an automorphism $g:N\imp N$ such that $g(a'\,b')=a\,b$. We conclude that $\psi(a,b)\wedge\neg p_{\restriction1}(b)$ holds in $(N,\sigma^g)$. 

Now we prove that $X_b$ is open in $Y$. Let $\sigma\in X_b$. We shall show that $\sigma$ belongs to a basic open set  contained in $X_b$. If $(N,\sigma)\models\neg p_{\restriction\A}(b)$ the claim is obvious, so suppose that $(N,\sigma)\models\neg\phi(b)$ for some  existential formula $\phi(x) \in p_{\restriction\E}(x)$. The expansions in $Y$ are existentially closed, hence (see, for instance, Theorem~7.2.4 in~\cite{hodges}) there is an existential formula $\psi(x)$ with $(N,\sigma) \models \psi(b)$, such that $\psi(x)\imp \neg\phi(x)$ holds for every $\tau\in Y$. Then $[\psi(b)]_N\subseteq X_b$ as required.\end{proof}

\begin{cor} Suppose that $T$ is $\omega$-categorical, $N$ is a countable model of $T$ and that $T_0$ has a model companion $\Tmc$ which is a complete theory. Then  an expansion $\alpha \in \Exp(N,T_0)$ is Truss generic if and only if it is an atomic model of $\Tmc$. \end{cor}

Theorem 3.9 is related to Theorem 4.2.6 in~\cite{hodges1} and to Theorem 2.4 in~\cite{ivanov}. Theorem~\ref{trussexistence} below is incidental to the main motivation of this paper and it gives a necessary and sufficient condition for Truss generic expansions to exist under the assumptions on $T$ and $T_0$ underlying this section. As remarked by the anonymous referee, in the $\omega$-categorical case Theorem~\ref{trussexistence} follows from Theorems 1.2, 1.3 and 2.4 \cite{ivanov}. In particular, conditions 2 and 3 are equivalent to JEP and AAP in \cite{ivanov}.

\begin{thm}\label{trussexistence} Suppose that $N \models T$ is countable and saturated and that all e.c. models of $T_0$ have the same existential theory. 
The following are equivalent:
\begin{itemize}
\item[1.] Truss generic expansions of $N$ exist;
\item[2.] for every finite $b$, the isolated points are dense in $D_b$;
\item[3.] for every finite $x$, the isolated points are dense in $S_x^Y$.
\end{itemize}
\end{thm}

\begin{proof} The equivalence $2\IFF3$ is clear by Lemma~\ref{homeomorphismDS}. Since the existence of e-atomic models implies that isolated points are dense in $S_x^Y$, the implication $1\IMP3$ follows from Theorem~\ref{trussatomicrich}. To prove the converse we assume 2 and construct a set $\Delta$ which is the quantifier-free diagram of an e-atomic model.

The diagram $\Delta$ is defined by finite approximations. Assume that at stage $i$ we have a finite set $\Delta_i$ of  quantifier-free sentences with parameters in $N$ which is consistent with $T_0$. Below we define $\Delta_{i+1}$. The definition uses a fixed arbitrary enumeration of length $\omega$ of all types of the form $p_{\restriction L}(x)\cup\{\phi(x)\}$ with finitely many parameters in $N$ and where $\phi(x)$ is quantifier-free. This exists because $T$ is small by assumption.

If $i$ is even, consider the $i/2$-th type in the given enumeration. If this type is consistent with $T_0\cup\Delta_i$, let $c$ be such that $T_0\cup p_{\restriction L}(c)\cup\{\phi(c)\}$ holds for some expansion and define $\Delta_{i+1}:=\Delta_i\cup\{\phi(c)\}$. Otherwise let $\Delta_{i+1}:=\Delta_i$. If $i$ is odd, let $b$ be a tuple that enumerates all the parameters in $\Delta_i$. Recall that we have assumed 2, so there is an expansion $\alpha$ which models $\Delta_i$ and is such that $\dg_{\restriction1}(\alpha,b)$ is isolated in $D_b$, say by the type $\E y\,[p_{\restriction L}(b,y)\wedge\phi(b,y)]$ where $\phi(b,y)$ is quantifier-free. Let $a$ satisfy $p_{\restriction L}(b,x)\wedge\phi(b,x)$ and define $\Delta_{i+1}:=\Delta_i\cup\{\phi(b,a)\}$.

Let $(N,\alpha)$ be the model with diagram $\Delta$. We claim that even stages guarantee both smoothness and existential closure. Smoothness is clear. To prove existential closure observe that if $\phi(x)$ is a quantifier-free formula with parameters in $N$ that has a solution in some extension of $(N,\alpha)$, then in particular it is consistent with $T_0\cup\Delta_i$ for every $i$, so at some stage $\phi(c)$ is added to the diagram of $(N,\alpha)$. Odd stages ensure that every type $p_{\restriction 1}(x)$ realized in $(N,\alpha)$ is e-isolated, so 1 follows by Theorem~\ref{trussatomicrich}.\end{proof}

\begin{exa}\label{automorphismsrandomgraph}
Truss generic automorphisms of the random graph.  Let $L$ be the language of graphs and let $T$ be the theory of the random graph. Let $L_0$ and $T_0$ be as in Example~\ref{PAPA}.  The existence of Truss generic automorphisms of the random graph was first proved in~\cite{truss} and extended to generic tuples in~\cite{hhls}, essentially using~\cite{hrush2}. These proofs use amalgamation properties of finite structures.

In the case of the random graph we can give a precise description of the isolated tuples. It is known~\cite{kikyo} that $T_0$ has no model companion. However, since the class of e.c. models of $T_0$ has the joint embedding property, all e.c. models have the same existential theory, hence $T$ and $T_0$ satisfy the hypothesis of Theorem~\ref{trussexistence}. The existence of Truss generic automorphisms of the random graph follows by the proposition below and Theorem~\ref{trussexistence}. This proof is by no means shorter than the one in~\cite{hhls}, and it still uses~\cite{hrush2}.

\begin{proposition} Let $T$ be the theory of the random graph and let $N$ be a countable random graph. Let $L_0$ and $T_0$ be as in  Example~\ref{PAPA}. Then for every finite tuple $b$ in $N$, the e-isolated points in $D_b$ are dense.\end{proposition}

\begin{proof}By the main result in~\cite{hrush2}, for every finite subset $B$ of  the random graph $N$ there is a finite set $A$ such that $B\subseteq A\subseteq N$ and every partial isomorphism $g:N\imp N$ with $\dom g, \range g\subseteq B$ has an extension to an automorphism of  $A$.

Let $\psi(b)$ be any existential formula consistent with $T_0$. Let $(N,\alpha)$ be a model that realizes $\psi(b)$. We shall show that $[\psi(b)]_D$ contains an isolated point. By the result in~\cite{hrush2} mentioned above, there is a model $(N,\sigma)$ which has a finite substructure $(A,\sigma\restriction A)$ that models $\psi(b)$. We may assume that $\sigma$ is existentially closed. Let $\phi(a,b)$ be the quantifier-free diagram of $A$ in $(N,\sigma)$. We claim that $\E z\,\phi(z,b)$ isolates a point of $D_b$, namely $\dg_{\restriction1}(\sigma,b)$.

To prove the claim, let $\tau\in Y$ model $\E z\,\phi(z,b)$ and prove that $(N,\tau)\equiv_{1,b}(N,\sigma)$. As $\phi(a,b)$ is the diagram of a substructure we can assume  that $(N,\tau)$ and $(N,\sigma)$ overlap on $A$. Since both $\sigma$ and $\tau$ are existentially closed and can be amalgamated over $A$, they are $1$-elementarily equivalent.\end{proof}
\end{exa}

\begin{exa}\label{autnocyles} Cycle-free automorphisms of the random graph. Let $L$, $T$, $N$, and $L_0$ be as in Example~\ref{automorphismsrandomgraph}.  The theory $T_0$ says that $f$ is an automorphism with inverse $f^{-1}$, and moreover for every positive integer $n$ it contains  the axiom $\A x\,f^nx\neq x$. These axioms claim that $f$ has no finite cycles. It is known~\cite{kumac} that $T_0$ has a model companion. Now we prove that there is no Truss generic expansion in $\Exp(N, T_0)$.

Suppose for a contradiction that some expansion $(N, \tau)$ is Truss generic. Let $b$ be an element of $N$. As $T$ is $\omega$-categorical, existential quasifinite formulas are equivalent to existential formulas. So, by Theorem~\ref{trussexistence}, there is an existential formula $\phi(b)$ that isolates $\dg_{\restriction1}(\tau, b)$ in $D_b$. As the symbol $f^{-1}$ can be eliminated at the cost of a few extra existential quantifiers, we may assume that it does not occur in $\phi(b)$. Let $n$ be a positive integer which is larger than the number of occurrences of the symbol $f$ in $\phi(b)$. Denote by $f_\tau$ the interpretation of $f$ in $(N,\tau)$. Let $A\subseteq N$ be a finite set containing $b$ and such that the sets $\{c,f_\tau c,\dots,f_\tau^{n-1}c\}$, for $c\in A$, are pairwise disjoint and let $B$ be the union of all these sets. Clearly we can choose $A$ such that $B$ contains witnesses of all the existential quantifiers in $\phi(b)$. The latter requirement guarantees that if $\alpha$ is an expansion such that $\alpha\restriction B=\tau\restriction B$, then $(N,\alpha)\models\phi(b)$. Define $d:=f_\tau^nb$ and $e:=f_\tau d$. Let $e'\in N$ realize the type $\tp_{\restriction L}(e/f_\tau[B])$ and be such that $r(b,e)\nleftrightarrow r(b,e')$. As $b\notin f_\tau[B]$, the theory of the random graph ensures the existence of such an $e'$. Let $g:=f_\tau\restriction B \cup\{\langle d,e'\rangle\}$. We claim that $g:N\imp N$ is a partial isomorphism. To prove the claim it suffices to check that $r(a,d)\iff r(ga,e')$ for every $a\in B$. We know that $r(a,d)\iff r(ga,e)$. As $ga\in f_\tau[B]$, by the choice of $e'$ we have $r(ga,e)\iff r(ga,e')$. Then $r(a,d)\iff r(ga,e')$ follows. Finally, it is easy to see that the homogeneity of $N$ yields an extension of $g$ to a cycle-free automorphism of $N$, hence an expansion $\alpha$. By construction, $\alpha\restriction B=\tau\restriction B$ so, as observed above, $(N,\alpha)\models\phi(b)$. But $(N,\tau)$ and $(N,\alpha)$ disagree on the truth of $r(b,f^{n+1}b)$. This contradicts that $\phi(b)$ isolates $\dg_{\restriction1}(\tau, b)$.\end{exa}

Example~\ref{autnocyles} shows that the existence of the model companion of $T_0$ is not sufficient to guarantee the existence of Truss generic expansions. The following corollary of Theorem~\ref{trussexistence} gives a sufficient condition.

\begin{cor}\label{saturatedcountablerichatomic}  Suppose that $T_0$ has a complete model companion $\Tmc$ which is small. Then $N$ has a Truss generic expansion.
\end{cor}

\begin{proof} Modulo $\Tmc$ every formula is equivalent to an existential (or, equivalently, to a universal) one. Then $S_x^Y$ is the set of all complete parameter-free types consistent with $\Tmc$. Though the topology on $S_x^Y$ is not the standard one, the usual argument (e.g.\@ Theorem 4.2.11 of~\cite{marker}) suffices to prove that the isolated types are dense.\end{proof}

\end{document}